\documentclass[10pt,a4paper]{article}
\usepackage{comment}
\usepackage{empheq}
\usepackage[T1]{fontenc}
\usepackage[mathscr]{euscript}
\usepackage{amsmath}
\usepackage{stmaryrd}
\usepackage{mathtools}
\usepackage{extarrows}
\usepackage{amsbsy}
\usepackage{amsfonts}
\usepackage{amssymb}
\usepackage{graphicx}
\usepackage{bbm}
\usepackage{listings}
\usepackage{color}
\usepackage[table]{xcolor}
\usepackage{xcolor}
\usepackage[left=1cm,right=2cm,top=2cm,bottom=2cm]{geometry}
\usepackage[latin1]{inputenc}
\usepackage{amsmath,amssymb,amsfonts}
\usepackage{enumerate}
\usepackage{yhmath}
\usepackage{amsthm}
\usepackage{picture}
\usepackage{colortbl}
\usepackage{hyperref}
\usepackage{wrapfig} 
\usepackage{graphicx}
\usepackage{xfrac}  
\usepackage{faktor}  

\usepackage{pst-all} 
\usepackage{pstricks}
\usepackage{pst-plot}
\usepackage{pst-text}
\usepackage{pstricks-add}
\usepackage{pstricks}
\usepackage{pst-eucl}

\usepackage{permute} 

\usepackage{array}

\newcommand{\lastcfrac}[2]{%
  \vphantom{\cfrac{#1}{#2}}%
  \raisebox{\dimexpr1ex-\height}{%
    $\displaystyle
      \raisebox{.5\height}{$\ddots$}+\cfrac{#1}{#2}
    $%
  }%
}

\usepackage{tikz}
\usepackage{relsize}

\usepackage{comment} 

\usepackage{dsfont} 
\usepackage{amsmath}
\usepackage{amsfonts}
\usepackage{amssymb}

\definecolor{ashgrey}{rgb}{0.7, 0.75, 0.71}

\usetikzlibrary{arrows}
\usetikzlibrary{arrows,shapes,positioning}
\usetikzlibrary[patterns]
\usetikzlibrary{decorations.markings}
\usepackage{pgf,tikz}
\usepackage{pgfplots}
\tikzstyle directed=[postaction={decorate,decoration={markings,
    mark=at position .5 with {\arrow[]{stealth}}}}]
\usepackage{tikz-cd} 
\usetikzlibrary{cd}
\usetikzlibrary{calc}
\usetikzlibrary{trees,positioning,arrows,fit,shapes}
\usetikzlibrary{matrix} 
\usetikzlibrary{angles,quotes,calc}
\definecolor{gris-clair}{rgb}{0.8,0.8,0.8}

\tikzstyle directed=[postaction={decorate,decoration={markings, 
mark=at position .65 with {\arrow{latex}}}}]

\usepackage{tikz-cd}

\tikzset{
  symbol/.style={
    draw=none,
    every to/.append style={
      edge node={node [sloped, allow upside down, auto=false]{$#1$}}}
  }
}

\usetikzlibrary{arrows,calc}

\usepackage[all]{xy}
\def\dar[#1]{\ar@<1pt>[#1]\ar@<-1pt>[#1]}
\entrymodifiers={!!<0pt,0.7ex>+}

\usepackage{amsthm} 
\newtheorem{theorem}{Theorem}[section]

\newtheorem{lemma}[theorem]{Lemma}

\theoremstyle{definition}

\theoremstyle{definition}
\title{An explicit decomposition formula of a matrix in $GL_{2}\left(\mathbb{Z}\right)$}
\author{D. FOSSE, MSc. Physics\\
dominique.fosse@a3.epfl.ch}\date{\vspace{-5ex}}
\begin{document}
\maketitle
\paragraph{Introduction}
Let $\mathcal{M}_{2}\left(\mathbb{Z}\right)$ the ring of all square matrices of order $2$ with coefficients in the ring $\mathbb{Z}$. Recall that $GL_{2}\left(\mathbb{Z}\right)$ denotes the unit group of $\mathcal{M}_{2}\left(\mathbb{Z}\right)$ and has the following caracterization:
\begin{equation*}
GL_{2}\left(\mathbb{Z}\right)=\left\{M\in\mathcal{M}_{2}\left(\mathbb{Z}\right)\big|\,\mathrm{det}(M)=\pm 1\right\}
\end{equation*}
We will make use of $C:=\left(\begin{smallmatrix}1&0\\0&-1\end{smallmatrix}\right)\in GL_{2}\left(\mathbb{Z}\right)$. Let's consider now
\begin{equation*}
SL_{2}\left(\mathbb{Z}\right)=\left\{M\in\mathcal{M}_{2}\left(\mathbb{Z}\right)\big|\,\mathrm{det}(M)=+ 1\right\}
\end{equation*}
which is a subgroup of $GL_{2}\left(\mathbb{Z}\right)$; we define $A:=\left(\begin{smallmatrix}1&1\\0&1\end{smallmatrix}\right)$ and $B:=\left(\begin{smallmatrix}1&0\\1&1\end{smallmatrix}\right)$ two elements of $SL_{2}\left(\mathbb{Z}\right)$. It is well known (for instance, see \cite{Kassel}) that $A$ and $B$ generates $SL_{2}\left(\mathbb{Z}\right)$; and from now on, we will use the following notation: 
\begin{equation*}
\langle A,B\rangle=SL_{2}\left(\mathbb{Z}\right)
\end{equation*}
Other pairs of generators can be considered; one can often find in the literature:
\begin{equation*}
S:=B^{-1}A B^{-1}=\left(\begin{smallmatrix}0&1\\-1&0\end{smallmatrix}\right)\qquad\textnormal{and}\qquad T:=B
\end{equation*}
Let $M:=\left(\begin{smallmatrix}a&b\\c&d\end{smallmatrix}\right)\in GL_{2}\left(\mathbb{Z}\right)$ and suppose $d\neq 0$ (the case $d=0$ is elementary and will be treated separately). The aim of this article is to demonstrate, using a funny induction, the following formula:
\begin{equation}\label{eq-1}
\boxed{
M=\left(A B^{-1}A\right)^{1-(-1)^{\left\lfloor\frac{j}{2}\right\rfloor}\mathrm{sgn}(d)}A\Big(\prod_{k=1}^{j}A^{-\left(2+(-1)^{k}n_{k}\right)}B\Big)\left(C A^{2}\right)^{\frac{1-\mathrm{det}(M)}{2}}A^{(-1)^{j}\mathrm{sgn}(d)\left(p_{j-1}c-q_{j-1}a\right)}B^{-1}A}
\end{equation}
Here, $\left[n_{1};n_{2},\ldots ,n_{j}\right]$ represents the simple finite continued fraction associated to the rational $\frac{b}{d}$; where $n_{1}\in\mathbb{Z}$ and $n_{i}\in\mathbb{N}^{\ast}$, $\forall i\in\llbracket 2,\, j\rrbracket$. Since $\left[n_{1};1\right]=\left[n_{1}+1\right]$ and $\left[n_{1};n_{2},\ldots ,n_{j},1\right]=\left[n_{1};n_{2},\ldots ,n_{j}+1\right]$, every rational number can be represented in two different ways and we will show that formula \eqref{eq-1} is independant of this choice of representation. The terms $p_{j-1}$ and $q_{j-1}$ come from the reduced fraction $\frac{p_{j-1}}{q_{j-1}}:=\left[n_{1};n_{2},\ldots ,n_{j-1}\right]$ with the initial condition $\left(p_{0},q_{0}\right):=(1,0)$. By definition of $\left[n_{1};n_{2},\ldots ,n_{j}\right]$, one has:
\begin{equation}\label{eq-1-0}
\frac{p_{j}}{q_{j}}=\frac{b}{d}\iff p_{j}d-q_{j}b=0
\end{equation}
Also, $\left\lfloor\frac{j}{2}\right\rfloor$ denotes the integer part of $\frac{j}{2}$ so that $(-1)^{\left\lfloor\frac{j}{2}\right\rfloor}=\pm 1$, depending on the residue of $j$ modulo $4$. If we note $I:=\left(\begin{smallmatrix}1&0\\0&1\end{smallmatrix}\right)$, then we verify by direct calculation that $\left(A B^{-1}A\right)^{2}=-I$; therefore:
\begin{equation}\label{factor}
\left(A B^{-1}A\right)^{1-(-1)^{\left\lfloor\frac{j}{2}\right\rfloor}\mathrm{sgn}(d)}=\left\{
\begin{array}{cr}
\left(A B^{-1}A\right)^{0}=I&\textnormal{if }(-1)^{\left\lfloor\frac{j}{2}\right\rfloor}\mathrm{sgn}(d)=+1\\
\left(A B^{-1}A\right)^{2}=-I&\textnormal{if }(-1)^{\left\lfloor\frac{j}{2}\right\rfloor}\mathrm{sgn}(d)=-1
\end{array}
\right.
\end{equation}
As $\left(A B^{-1}A\right)^{1-(-1)^{\left\lfloor\frac{j}{2}\right\rfloor}\mathrm{sgn}(d)}=\pm I$, this matrix commutes with any element of $GL_{2}\left(\mathbb{Z}\right)$ and we chose to write it as a factor of the right member of formula \eqref{eq-1}. The basic theory of continued fractions also ensures that $q_{k}>0$, $\forall k\in\llbracket 1,\,j\rrbracket$ and so there is no ambiguity regarding the sign of $p_{j-1}$ in case the ratio $\frac{p_{j-1}}{q_{j-1}}$ is negative. Note that $\mathrm{det}\left(M\right)=+1\iff M\in SL_{2}\left(\mathbb{Z}\right)$, then $\left(C A^{2}\right)^{\frac{1-\mathrm{det}(M)}{2}}=\left(C A^{2}\right)^{0}=I$ which means, as expected, that $C$ (which doesn't belong to $SL_{2}\left(\mathbb{Z}\right)$) vanishes from formula \eqref{eq-1} and we retrieve an expression of $M$ as a word in $\langle A,B\rangle$.
\paragraph{An explicit example}
\begin{enumerate}[$1)$]
\item Let $M:=\left(\begin{smallmatrix}-65&17\\42&-11\end{smallmatrix}\right)$; we verify that $\mathrm{det}\left(M\right)=1$ so that $M\in SL_{2}\left(\mathbb{Z}\right)$. We develop here what we call the \emph{first represention} of $\frac{b}{d}=-\frac{17}{11}$ which is $\left[-2;2,5\right]$. Explicitely, 
\begin{equation*}
-\frac{17}{11}= -2+\cfrac{1}{2+\cfrac{1}{5}}\implies j:=3\textnormal{ and }\left(n_{1},n_{2},n_{3}\right)=\left(-2,2,5\right)
\end{equation*}
Then, $(-1)^{\left\lfloor\frac{j}{2}\right\rfloor}\mathrm{sgn}(d)=(-1)^{\left\lfloor\frac{3}{2}\right\rfloor}\mathrm{sgn}(-11)=(-1)^{1}(-1)=+1$. The reduced fraction $\frac{p_{j-1}}{q_{j-1}}=\frac{p_{2}}{q_{2}}$ is then $\left[-2;2\right]=-2+\frac{1}{2}=-\frac{3}{2}$. As stated in the introduction, $q_{2}$ is necessarily a positive integer; thus $\left(p_{2},q_{2}\right)=\left(-3,2\right)$. Then $b_{j}=b_{3}=(-1)^{3}\mathrm{sgn}(-11)\left(-3\cdot 42-2\cdot (-65)\right)=4$. Also, $\mathrm{det}(M)=1\implies\frac{1-\mathrm{det}(M)}{2}=0\implies\left(C A^{2}\right)^{\frac{1-\mathrm{det}(M)}{2}}=\left(C A^{2}\right)^{0}=I$. That's it; we have everything to apply formula \eqref{eq-1}:
\begin{align}\label{eq-2}
M&=I\cdot A\big(A^{-\left(2-n_{1}\right)}B\big)\big(A^{-\left(2+n_{2}\right)}B\big)\big(A^{-\left(2-n_{3}\right)}B\big)I\cdot A^{b_{3}}B^{-1}A\nonumber\\
&=A\big(A^{-\left(2-(-2)\right)}B\big)\big(A^{-\left(2+2\right)}B\big)\big(A^{-\left(2-5\right)}B\big)A^{4}B^{-1}A\nonumber\\
&=A^{-3}B A^{-4}B A^{3}B A^{4}B^{-1}A
\end{align}
\item Let's consider the same matrix $M:=\left(\begin{smallmatrix}-65&17\\42&-11\end{smallmatrix}\right)$ but this time, let's use the \emph{second representation} of $\frac{b}{d}=-\frac{17}{11}$ which is $\left[-2;2,4,1\right]\implies\left(n_{1},n_{2},n_{3},n_{4}\right)=(-2,2,4,1)$. This time, $j:=4$ and thus $(-1)^{\left\lfloor\frac{j}{2}\right\rfloor}\mathrm{sgn}(d)=(-1)^{\left\lfloor\frac{4}{2}\right\rfloor}\mathrm{sgn}(-11)=(-1)^{2}(-1)=-1$. The reduced fraction $\frac{p_{j-1}}{q_{j-1}}=\frac{p_{3}}{q_{3}}$ is then $\left[-2;2,4\right]=-2+\frac{1}{2+\frac{1}{4}}=-\frac{14}{9}\implies\left(p_{j-1},q_{j-1}\right)=\left(p_{3},q_{3}\right)=(-14,9)$. Then, $b_{j}=b_{4}=(-1)^{4}\mathrm{sgn}\left(-11\right)\left((-14)42-9(-65)\right)=3$. Then, 
\begin{align}\label{eq-3}
M&=\left(A B^{-1}A\right)^{2}A\big(A^{-\left(2-n_{1}\right)}B\big)\big(A^{-\left(2+n_{2}\right)}B\big)\big(A^{-\left(2-n_{3}\right)}B\big)\big(A^{-\left(2+n_{4}\right)}B\big)A^{b_{4}}B^{-1}A\nonumber\\
&=\left(A B^{-1}A\right)^{2}A\big(A^{-\left(2-(-2)\right)}B\big)\big(A^{-\left(2+2\right)}B\big)\big(A^{-\left(2-4\right)}B\big)\big(A^{-\left(2+1\right)}B\big)A^{3}B^{-1}A\\
&=A B^{-1}A^{2}B^{-1}A^{2}A^{-4}B A^{-4}B A^{2}B A^{-3}B A^{3}B^{-1}A\nonumber\\
&=A B^{-1}A^{2}B^{-1}A^{-2}B A^{-4}B A^{2}B A^{-3}B A^{3}B^{-1}A
\end{align}
\end{enumerate}
Comparing \eqref{eq-2} and \eqref{eq-3}, we get two different expressions of $M$ in $\langle A,B\rangle$ and formula \eqref{eq-1} works well in both representations.
\paragraph{Some basic lemmas}
We list here all the requiered results used in the demonstration of formula \eqref{eq-1}.
\begin{lemma}[Powers of $A$ and $B$]\label{lemma-1}
For all $n\in\mathbb{Z}$, 
\begin{equation}\label{lemma-1-1-eq}
A^{n}=\begin{pmatrix}
1&n\\
0&1
\end{pmatrix}\qquad\qquad
B^{n}=\begin{pmatrix}
1&0\\
n&1
\end{pmatrix}
\end{equation}
\end{lemma}
\begin{proof}
Suppose $n\geq 0$. For $n=0$ or $n=1$, \eqref{lemma-1-1-eq} are both verified. Suppose \eqref{lemma-1-1-eq} true for $n>1$; one gets $\left(\begin{smallmatrix}1&1\\0&1\end{smallmatrix}\right)A^{n}=\left(\begin{smallmatrix}1&1\\0&1\end{smallmatrix}\right)\left(\begin{smallmatrix}1&n\\0&1\end{smallmatrix}\right)=\left(\begin{smallmatrix}1&n+1\\0&1\end{smallmatrix}\right)=A^{n}\left(\begin{smallmatrix}1&1\\0&1\end{smallmatrix}\right)\implies A\cdot A^{n}=A^{n}\cdot A=A^{n+1}$. Regarding $B$, we have $B\cdot B^{n}=\left(\begin{smallmatrix}1&0\\1&1\end{smallmatrix}\right)\left(\begin{smallmatrix}1&0\\n&1\end{smallmatrix}\right)=\left(\begin{smallmatrix}1&0\\n+1&1\end{smallmatrix}\right)=B^{n+1}=B^{n}\cdot B$. Now let's compute the inverse of $A^{n}$: $\left(A^{n}\right)^{-1}=\left(\begin{smallmatrix}1&n\\0&1\end{smallmatrix}\right)^{-1}=\left(\begin{smallmatrix}1&-n\\0&1\end{smallmatrix}\right)=A^{-n}$ and we get something similar for $B$: $B^{-n}=\left(\begin{smallmatrix}1&0\\-n&1\end{smallmatrix}\right)$ which proves \eqref{lemma-1-1-eq}, $\forall n\in\mathbb{Z}$.
\end{proof}
Let's now treat the case $d:=0$ separately.
\begin{lemma}[The case $d:=0$]\label{lemma-2}
Let $M_{0}:=\left(\begin{smallmatrix}a&b\\c&0\end{smallmatrix}\right)\in GL_{2}\left(\mathbb{Z}\right)$, then $M_{0}\in\langle A,B,C\rangle$
\end{lemma}
\begin{proof}
$M_{0}=\left(\begin{smallmatrix}a&b\\c&0\end{smallmatrix}\right)\implies\mathrm{det}\left(M_{0}\right)=-bc=\pm 1$. Thus, there are four possibilities:
\begin{equation*}
(b,c)\in\left\{(1,1),(-1,-1),(1,-1),(-1,1)\right\}
\end{equation*}
\begin{enumerate}[$(i)$]
\item $(b,c):=(1,1)\implies M_{0}=\left(\begin{smallmatrix}a&1\\1&0\end{smallmatrix}\right)\implies M_{0}\in GL_{2}\left(\mathbb{Z}\right)\setminus SL_{2}\left(\mathbb{Z}\right)$ as $\mathrm{det}\left(M_{0}\right)=-1$. We check that, $\forall a\in\mathbb{Z}$:
\begin{equation}\label{1}
C B^{-1}A B^{a-1}=\left(\begin{smallmatrix}a&1\\1&0\end{smallmatrix}\right)\in\langle A,B,C\rangle
\end{equation}
\item $(b,c):=(-1,-1)\implies M_{0}=\left(\begin{smallmatrix}a&-1\\-1&0\end{smallmatrix}\right)\implies M_{0}\in GL_{2}\left(\mathbb{Z}\right)\setminus SL_{2}\left(\mathbb{Z}\right)$ as $\mathrm{det}\left(M_{0}\right)=-1$. Note that $M_{0}=-\left(\begin{smallmatrix}-a&1\\1&0\end{smallmatrix}\right)$. Using $-I=(A B^{-1}A)^{2}$ as mentioned in the introduction and point $(i)$, we get:
\begin{equation}\label{2}
M_{0}=A B^{-1}A^{2}B^{-1}A C B^{-1}A B^{-a-1}\in\langle A,B,C\rangle
\end{equation}
\item $(b,c)=(1,-1)\implies M_{0}=\left(\begin{smallmatrix}a&1\\-1&0\end{smallmatrix}\right)\implies \mathrm{det}\left(M_{0}\right)=+1\implies M_{0}\in SL_{2}\left(\mathbb{Z}\right)$. We check that, $\forall a\in\mathbb{Z}$:
\begin{equation}\label{3}
M_{0}=A^{1-a}B^{-1}A\in\langle A,B\rangle\subseteq\langle A,B,C\rangle
\end{equation}
\item $(b,c)=(-1,1)\implies M_{0}=\left(\begin{smallmatrix}a&-1\\1&0\end{smallmatrix}\right)\implies \mathrm{det}\left(M_{0}\right)=+1\implies M_{0}\in SL_{2}\left(\mathbb{Z}\right)$. We check that, $\forall a\in\mathbb{Z}$:
\begin{equation}\label{4}
M_{0}=B A^{-1} B^{1-a}\in\langle A,B\rangle\subseteq\langle A,B,C\rangle
\end{equation}
\end{enumerate}
Conclusion: as per equations \eqref{1}, \eqref{2}, \eqref{3} and \eqref{4}, $M_{0}\in\langle A,B,C\rangle$. On top of that, equations \eqref{3} and \eqref{4} show that $M_{0}\in SL_{2}\left(\mathbb{Z}\right)\implies M_{0}\in\langle A,B\rangle$, as expected.
\end{proof}
\begin{lemma}[Some basic results on simple continued fractions]\label{lemma-3}
Let $\left[n_{1};n_{2},\ldots ,n_{j}\right]$ a simple and finite continued fraction:
\begin{equation*}
\frac{p_{j}}{q_{j}}=\left[n_{1};n_{2},\ldots ,n_{j}\right]=n_{1}+\cfrac{1}{n_{2} +\cfrac{1}{n_{3} +\cfrac{1}{n_{4} + \lastcfrac{1}{n_{j-1}+\frac{1}{n_{j}}}}}}
\end{equation*}
The \emph{convergents} are the rational numbers defined by $\frac{p_{i}}{q_{i}}:=\left[n_{1};n_{2},\ldots ,n_{i}\right]$, $\forall i\in\llbracket 1,\,j\rrbracket$ with the convention $\left(p_{0},q_{0}\right):=(1,0)$. Let's prove the following points:
\begin{enumerate}[$(i)$]
\item $\forall i\in\llbracket 2,\,j\rrbracket$, we have $p_{i}=n_{i}p_{i-1}+p_{i-2}$ and $q_{i}=n_{i}q_{i-1}+q_{i-2}$
\item $p_{i}q_{i-1}-p_{i-1}q_{i}=(-1)^{i}$, $\forall i\in\llbracket 1,\,j\rrbracket$
\item The convergents $\frac{p_{i}}{q_{i}}:=\left[n_{1};n_{2},\ldots ,n_{i}\right]$ are such that $p_{i}$ and $q_{i}$ are coprime numbers, $\forall i\in\llbracket 1,\,j\rrbracket$.
\item With $q_{0}:=0$, one has $q_{1}:=1\leq q_{2}$ and $q_{2}<q_{3}<\ldots q_{j}$. In particular, $q_{i}\geq 0$, $\forall i\in\llbracket 0,\,j\rrbracket$.
\item $\frac{p_{i}}{q_{i}}-\frac{p_{i-1}}{q_{i-1}}=\frac{(-1)^{i}}{q_{i}q_{i-1}}$, $\forall i\in\llbracket 2,\,j\rrbracket$
\end{enumerate}
\end{lemma}
\begin{proof}
\begin{enumerate}[$(i)$]
\item As $p_{0}=1$, $q_{1}=1$ and $\frac{p_{1}}{q_{1}}=\left[n_{1}\right]=\frac{n_{1}}{1}=n_{1}$, we have $p_{1}:=n_{1}$. Then, $n_{2}p_{1}+p_{0}=n_{2}n_{1}+1$. On the other side, $\frac{p_{2}}{q_{2}}=\left[n_{1};n_{2}\right]=n_{1}+\frac{1}{n_{2}}=\frac{n_{1}n_{2}+1}{n_{2}}\implies\left(p_{2},q_{2}\right)=\left(n_{1}n_{2}+1,n_{2}\right)$ and this shows that $(i)$ is valid for $i:=2$. Suppose that $(i)$ is valid for $i>2$; we have:
\begin{equation*}
\frac{p_{i}}{q_{i}}=\left[n_{1},n_{2},\ldots ,n_{i-1},n_{i}\right]=n_{1}+\cfrac{1}{n_{2} +\cfrac{1}{n_{3} +\cfrac{1}{n_{4} + \lastcfrac{1}{n_{i-1}+\frac{1}{n_{i}}}}}}
\end{equation*}
And we see directly that $\left[n_{1};n_{2},\ldots ,n_{i-1}+\frac{1}{n_{i}}\right]=\left[n_{1};n_{2},\ldots ,n_{i-1},n_{i}\right]$. Then,
\begin{align*}
\frac{p_{i}}{q_{i}}=\left[n_{1};n_{2},\ldots ,n_{i-1},n_{i}\right]&=\left[n_{1};n_{2},\ldots ,n_{i-1}+\tfrac{1}{n_{i}}\right]\\&=\frac{p_{i-1}\left(n_{1},n_{2},\ldots ,n_{i-1}+\tfrac{1}{n_{i}}\right)}{q_{i-1}\left(n_{1},n_{2},\ldots ,n_{i-1}+\tfrac{1}{n_{i}}\right)}\\
&=\frac{\left(n_{i-1}+\tfrac{1}{n_{i}}\right)p_{i-2}+p_{i-3}}{\left(n_{i-1}+\tfrac{1}{n_{i}}\right)q_{i-2}+q_{i-3}}\qquad\textnormal{(by inductive hypothesis)}\\
&=\frac{\left(n_{i-1}p_{i-2}+p_{i-3}\right)+\tfrac{1}{n_{i}}p_{i-2}}{\left(n_{i-1}q_{i-2}+q_{i-3}\right)+\tfrac{1}{n_{i}}q_{i-2}}\\
&=\frac{p_{i-1}+\tfrac{1}{n_{i}}p_{i-2}}{q_{i-1}+\tfrac{1}{n_{i}}q_{i-2}}\qquad\textnormal{(by inductive hypothesis)}\\
&=\frac{n_{i}p_{i-1}+p_{i-2}}{n_{i}q_{i-1}+q_{i-2}}
\end{align*}
\item For $i:=1$, $(ii)$ is verified, as $p_{1}q_{0}-p_{0}q_{1}=n_{1}\cdot 0-1\cdot 1=-1=(-1)^{1}$. Suppose $(ii)$ is true for $i>1$; one gets:
\begin{align*}
p_{i+1}q_{i}-p_{i}q_{i+1}&=\left(n_{i+1}p_{i}+p_{i-1}\right)q_{i}-p_{i}\left(n_{i+1}q_{i}+q_{i-1}\right)\qquad\textnormal{(using }(i))\\
&=n_{i+1}p_{i}q_{i}+p_{i-1}q_{i}-n_{i+1}q_{i}p_{i}-p_{i}q_{i-1}=-\left(p_{i}q_{i-1}-p_{i-1}q_{i}\right)\\
&=-(-1)^{i}\qquad\textnormal{(by inductive hypothesis)}\\
&=(-1)^{i+1}
\end{align*}
\item Both recurrence relations of point $(i)$ show that $n_{i}\in\mathbb{Z}\implies (p_{i},q_{i})\in\mathbb{Z}^{2}$, $\forall i\in\llbracket 1,\,j\rrbracket$. Let's write point $(ii)$ as $p_{i}\left((-1)^{i}q_{i-1}\right)+q_{i}\left((-1)^{i-1}p_{i-1}\right)=1$, $\forall i\in\llbracket 1,\,j\rrbracket$ which is a B\'ezout relation. Therefore, $p_{i}$ and $q_{i}$ are coprime numbers, $\forall i\in\llbracket 1,\,j\rrbracket$.
\item Using the recurrence relation $q_{i}=n_{1}q_{i-1}+q_{i-2}$, $\forall i\in\llbracket 2,\,j\rrbracket$ from point $(i)$ with $\left(q_{0},q_{1}\right)=\left(0,1\right)$, we show, by induction, that $q_{i}\geq 1$, $\forall i\in\llbracket 1,\,j\rrbracket$. Recall that $n_{1}\in\mathbb{Z}$ and $n_{i}\in\mathbb{N}^{\ast}$, $\forall i\in\llbracket 2,\,j\rrbracket$. For $i:=2$, we get $q_{2}=n_{2}q_{1}+q_{0}=n_{2}\cdot 1+0=n_{2}\geq 1$. Suppose that $q_{i}\geq 1$ for $i>2$, hence $q_{i+1}=n_{i+1}q_{i}+q_{i-1}$; by induction hypothesis, $q_{i-1}\geq 1$, $q_{i}\geq 1$ and $n_{i+1}\in\mathbb{N}^{\ast}$. Therefore, $n_{i+1}q_{i}+q_{i-1}\geq 1$; i.e, $q_{i+1}\geq 1$ and this shows that $q_{i}\geq 1$, $\forall i\in\llbracket 1,\,j\rrbracket$. Moreover, $n_{i}q_{i-1}+q_{i-2}\geq q_{i-1}+q_{i-2}$ when $i\geq 2$. Using point $(i)$, we get $q_{i}\geq q_{i-1}+q_{i-2}$, $\forall i\in\llbracket 2,\,j\rrbracket$. As $q_{i-2}\geq 1$ whenever $i\geq 3$, we get finally $q_{i}\geq q_{i-1}+q_{i-2}> q_{i-1}$, $\forall i\in\llbracket 3,\,j\rrbracket$.
\item Point $(iv)$ showed, in particular, that $q_{i}\neq 0$, $\forall i\in\llbracket 1,\,j\rrbracket$. Hence, $q_{i}q_{i-1}\neq 0$, $\forall i\in\llbracket 2,\,j\rrbracket$. It's then possible to divide point $(ii)$ relation by $q_{i}q_{i-1}$.
\end{enumerate}
\end{proof}
We will also make use of the following elementary result:
\begin{lemma}\label{lemma-4}
\begin{equation*}
(-1)^{\left\lfloor\frac{k+1}{2}\right\rfloor}=(-1)^{k}(-1)^{\left\lfloor\frac{k}{2}\right\rfloor}\qquad\forall k\in\mathbb{N}
\end{equation*}
\end{lemma}
\begin{proof}
Recall that $\forall x\in\mathbb{R}$ and $\forall n\in\mathbb{Z}$, one has $\left\lfloor x+n\right\rfloor =\left\lfloor x\right\rfloor +n$. Let $\left(k,k^{\prime}\right)\in\mathbb{N}^{2}$ such $k=4 k^{\prime}$, then $\left\lfloor\frac{k}{2}\right\rfloor =\left\lfloor\frac{4 k^{\prime}}{2}\right\rfloor =2 k^{\prime}\implies (-1)^{\left\lfloor\frac{k}{2}\right\rfloor}=(-1)^{2 k^{\prime}}=+1$. Suppose now $\left(k,k^{\prime}\right)\in\mathbb{N}^{2}$ such $k=4 k^{\prime}+1$; then $\left\lfloor\frac{k}{2}\right\rfloor =\left\lfloor\frac{4 k^{\prime}+1}{2}\right\rfloor =\left\lfloor 2 k^{\prime}+\frac{1}{2}\right\rfloor =2 k^{\prime}+\left\lfloor\frac{1}{2}\right\rfloor =2 k^{\prime}\implies (-1)^{\left\lfloor\frac{k}{2}\right\rfloor}=(-1)^{2 k^{\prime}}=+1$. Suppose now $\left(k,k^{\prime}\right)\in\mathbb{N}^{2}$ such $k=4 k^{\prime}+2$; then $\left\lfloor\frac{k}{2}\right\rfloor =\left\lfloor\frac{4 k^{\prime}+2}{2}\right\rfloor =\left\lfloor 2 k^{\prime}+1\right\rfloor =2 k^{\prime}+1\implies (-1)^{\left\lfloor\frac{k}{2}\right\rfloor}=(-1)^{2 k^{\prime}+1}=-1$. Finally, suppose $\left(k,k^{\prime}\right)\in\mathbb{N}^{2}$ such $k=4 k^{\prime}+3$; then $\left\lfloor\frac{k}{2}\right\rfloor =\left\lfloor\frac{4 k^{\prime}+3}{2}\right\rfloor =\left\lfloor\frac{\left(4 k^{\prime}+2\right)+1}{2}\right\rfloor =\left\lfloor 2 k^{\prime}+1+\frac{1}{2}\right\rfloor =2 k^{\prime}+1+\left\lfloor\frac{1}{2}\right\rfloor =2 k^{\prime}+1\implies (-1)^{\left\lfloor\frac{k}{2}\right\rfloor}=(-1)^{2 k^{\prime}+1}=-1$. Hence we showed that, $\forall k\in\mathbb{N}$:
\begin{equation}\label{lemma-4-proof-1-equ}
(-1)^{\left\lfloor\frac{k}{2}\right\rfloor}=\left\{
\begin{array}{cl}
1&\textnormal{if }k\equiv 0\textnormal{ or }1\mod{4}\\
-1&\textnormal{if }k\equiv 2\textnormal{ or }3\mod{4}
\end{array}
\right.\implies
(-1)^{\left\lfloor\frac{k+1}{2}\right\rfloor}=\left\{
\begin{array}{cl}
1&\textnormal{if }k\equiv 0\textnormal{ or }3\mod{4}\\
-1&\textnormal{if }k\equiv 1\textnormal{ or }2\mod{4}
\end{array}
\right.
\end{equation}
Of course, $\forall k\in\mathbb{N}$, we have:
\begin{equation*}
(-1)^{k}=\left\{
\begin{array}{cl}
1&\textnormal{if }k\equiv 0\textnormal{ or }2\mod{4}\\
-1&\textnormal{if }k\equiv 2\textnormal{ or }3\mod{4}
\end{array}
\right.
\end{equation*}
That means $(-1)^{k}(-1)^{\left\lfloor\frac{k}{2}\right\rfloor}$ equals $+1$ when $\left((-1)^{\left\lfloor\frac{k}{2}\right\rfloor},(-1)^{k}\right)=(1,1)$ or $(-1,-1)$ and this is the case if and only if $k\equiv 0$ or $3\mod{4}$ and this is exactly what shows equation \eqref{lemma-4-proof-1-equ}.
\end{proof}
\paragraph{The main result}
Let $M:=\left(\begin{smallmatrix}a&b\\c&d\end{smallmatrix}\right)\in GL_{2}\left(\mathbb{Z}\right)$ with $d\neq 0$. Let's define, $\forall k\in\llbracket 1,\,j\rrbracket$,
\begin{equation}\label{eq-4}
\left\{
\begin{array}{l}
\alpha_{k}(b,d):=(-1)^{\left\lfloor\frac{k}{2}\right\rfloor}\left(q_{k}b-p_{k}d+(-1)^{k}\left(q_{k-1}b-p_{k-1}d\right)\right)\\
\gamma_{k}(b,d):=(-1)^{\left\lfloor\frac{k}{2}\right\rfloor}\left(p_{k}d-q_{k}b\right)
\end{array}
\right.
\end{equation}
where, $\frac{p_{k}}{q_{k}}:=\left[n_{1};n_{2},\ldots ,n_{k}\right]$, $\forall k\in\llbracket 1,\,j\rrbracket$ are the convergents of the continued fraction $\frac{b}{d}=\left[n_{1};n_{2},\ldots ,n_{j}\right]$. Let's also define:
\begin{equation}\label{eq-5}
P_{0}:=A^{-1}M A^{-1}B\qquad\textnormal{and }
\qquad P_{k}:=\begin{pmatrix}
\alpha_{k}(b,d)&\alpha_{k}(b-a,d-c)\\
\gamma_{k}(b,d)&\gamma_{k}(b-a,d-c)
\end{pmatrix}\quad\forall k\geq 1
\end{equation}
Then, 
\begin{enumerate}[$(i)$]
\item $P_{k}\in GL_{2}\left(\mathbb{Z}\right)$, $\forall k\in\llbracket 0,\,j\rrbracket$ 
\item $P_{j}=(-1)^{\left\lfloor\frac{j}{2}\right\rfloor}\mathrm{sgn}(d)\left(C A^{2}\right)^{\frac{1-\mathrm{det}\left(M\right)}{2}}A^{b_{j}}$; where $b_{j}:=(-1)^{j}\mathrm{sgn}(d)\left(p_{j-1}c-q_{j-1}a\right)$
\item $P_{k}=B^{-1}A^{2+(-1)^{k}n_{k}}P_{k-1}$, $\forall k\in\llbracket 1,\,j\rrbracket$
\end{enumerate}
\begin{proof}
\begin{enumerate}[$(i)$]
\item For $k:=0$, it is clear, from its definition ($GL_{2}\left(\mathbb{Z}\right)$ is a group), that $P_{0}\in GL_{2}\left(\mathbb{Z}\right)$. Suppose $k>0$; from their definitions \eqref{eq-4}, we see that the coefficients of $P_{k}$ are integers. Therefore, the only thing we have to check is $\mathrm{det}\left(P_{k}\right)=\pm 1$, $\forall k\in\llbracket 1,\,j\rrbracket$. Let's do it:
\begin{align*}
\mathrm{det}\left(P_{k}\right)=&\,\alpha_{k}(b,d)\gamma_{k}(b-a,d-c)-\gamma_{k}(b,d)\alpha_{k}(b-a,d-c)\\
=&\,(-1)^{2\left\lfloor\frac{k}{2}\right\rfloor}\left(q_{k}b-p_{k}d+(-1)^{k}\left(q_{k-1}b-p_{k-1}d\right)\right)\left(p_{k}(d-c)-q_{k}(b-a)\right)\\
&-(-1)^{2\left\lfloor\frac{k}{2}\right\rfloor}\left(p_{k}d-q_{k}b\right)\left(q_{k}(b-a)-p_{k}(d-c)+(-1)^{k}\left(q_{k-1}(b-a)-p_{k-1}(d-c)\right)\right)\\
=&\,q_{k}p_{k}b(d-c)-q_{k}^{2}b(b-a)-p_{k}^{2}d(d-c)+p_{k}q_{k}d(b-a)+(-1)^{k}p_{k}q_{k-1}b(d-c)\\
&-(-1)^{k}q_{k}q_{k-1}b(b-a)-(-1)^{k}p_{k}p_{k-1}d(d-c)+(-1)^{k}p_{k-1}q_{k}d(b-a)-p_{k}q_{k}d(b-a)+p_{k}^{2}d(d-c)\\
&-(-1)^{k}p_{k}q_{k-1}d(b-a)+(-1)^{k}p_{k}p_{k-1}d(d-c)+q_{k}^{2}b(b-a)-q_{k}p_{k}b(d-c)+(-1)^{k}q_{k}q_{k-1}b(b-a)\\
&-(-1)^{k}q_{k}p_{k-1}b(d-c)\\
=&\,(-1)^{k}b(d-c)\left(p_{k}q_{k-1}-q_{k}p_{k-1}\right)-(-1)^{k}d(b-a)\left(p_{k}q_{k-1}-p_{k-1}q_{k}\right)\\
=&\,(-1)^{k}\left(p_{k}q_{k-1}-q_{k}p_{k-1}\right)(ad-bc)\\
=&\,(-1)^{k}(-1)^{k}\mathrm{det}\left(M\right)\qquad\left(\textnormal{using lemma \eqref{lemma-3}\textnormal{, point }}(ii)\right)\\
=&\,\mathrm{det}\left(M\right)\\
=&\,\pm 1\qquad\left(\textnormal{as }M\in GL_{2}\left(\mathbb{Z}\right)\right)
\end{align*}
\item Using \eqref{eq-1-0}, we get directly $\gamma_{j}(b,d)=0$ and this makes $P_{j}$ upper triangular. We have:
\begin{align}\label{eq-6}
\alpha_{j}(b,d)&=(-1)^{\left\lfloor\frac{j}{2}\right\rfloor}\left(q_{j}b-p_{j}d+(-1)^{j}\left(q_{j-1}b-p_{j-1}d\right)\right)\nonumber\\
&=(-1)^{\left\lfloor\frac{j}{2}\right\rfloor}(-1)^{j}\left(q_{j-1}b-p_{j-1}d\right)\qquad\left(\textnormal{using equation \eqref{eq-1-0}}\right)\nonumber\\
&=(-1)^{\left\lfloor\frac{j}{2}\right\rfloor}(-1)^{j}\left(q_{j-1}\left(d\frac{p_{j}}{q_{j}}\right)-p_{j-1}d\right)\qquad\left(\textnormal{using equation \eqref{eq-1-0} again}\right)\nonumber\\
&=(-1)^{\left\lfloor\frac{j}{2}\right\rfloor}(-1)^{j}\left(p_{j}q_{j-1}-p_{j-1}q_{j}\right)\frac{d}{q_{j}}\nonumber\\
&=(-1)^{\left\lfloor\frac{j}{2}\right\rfloor}(-1)^{j}(-1)^{j}\frac{d}{q_{j}}\qquad\left(\textnormal{using lemma \eqref{lemma-3}\textnormal{, point }}(ii)\right)\nonumber\\
&=(-1)^{\left\lfloor\frac{j}{2}\right\rfloor}\frac{d}{q_{j}}
\end{align}
From point $(i)$, we know that $P_{j}\in GL_{2}\left(\mathbb{Z}\right)$. Therefore, $\alpha_{j}(b,d)\in\mathbb{Z}$ with $d\in\mathbb{Z}^{\ast}$ and this means that $q_{j}$ divides $d$ (let's note this $q_{j}\mid d$).
Also, 
\begin{align}\label{eq-8}
\gamma_{j}(b-a,d-c)=&(-1)^{\left\lfloor\frac{j}{2}\right\rfloor}\left(p_{j}(d-c)-q_{j}(b-a)\right)=(-1)^{\left\lfloor\frac{j}{2}\right\rfloor}\left(p_{j}d-p_{j}c-q_{j}b+q_{j}a\right)\nonumber\\
=&(-1)^{\left\lfloor\frac{j}{2}\right\rfloor}\left(q_{j}a-p_{j}c\right)\qquad\left(\textnormal{using equation \eqref{eq-1-0}}\right)\nonumber\\
=&(-1)^{\left\lfloor\frac{j}{2}\right\rfloor}\left(q_{j}a-\left(\frac{q_{j}}{d}b\right)c\right)\qquad\left(\textnormal{using equation \eqref{eq-1-0} again}\right)\nonumber\\
=&(-1)^{\left\lfloor\frac{j}{2}\right\rfloor}\frac{q_{j}}{d}\left(ad-bc\right)\nonumber\\
=&(-1)^{\left\lfloor\frac{j}{2}\right\rfloor}\frac{q_{j}}{d}\mathrm{det}(M)
\end{align}
Using the same argument as for $\eqref{eq-6}$, we get $d\mid q_{j}$. So, as $q_{j}>0$ (that is lemma \eqref{lemma-3}, point $(iv)$), we have $\big(q_{j}\mid d$ and $d\mid q_{j}\big)\implies d=\mathrm{sgn}(d)q_{j}$. We have found:
\begin{equation}\label{eq-7}
\alpha_{j}(b,d)=(-1)^{\left\lfloor\frac{j}{2}\right\rfloor}\mathrm{sgn}(d)
\end{equation}
And,
\begin{equation}\label{eq-9}
\gamma_{j}(b-a,d-c)=(-1)^{\left\lfloor\frac{j}{2}\right\rfloor}\mathrm{sgn}(d)\mathrm{det}(M)
\end{equation}
Finally,
\begin{small}
\begin{align}\label{eq-10}
\alpha_{j}\left(b-a,d-c\right)=&\,(-1)^{\left\lfloor\frac{j}{2}\right\rfloor}\left(q_{j}(b-a)-p_{j}(d-c)+(-1)^{j}\left(q_{j-1}(b-a)-p_{j-1}(d-c)\right)\right)\nonumber\\
=&-\gamma_{j}(b-a,d-c)+(-1)^{\left\lfloor\frac{j}{2}\right\rfloor}(-1)^{j}\left(q_{j-1}(b-a)-p_{j-1}(d-c)\right)\nonumber\\
=&-\gamma_{j}(b-a,d-c)+(-1)^{\left\lfloor\frac{j}{2}\right\rfloor}(-1)^{j}\left(q_{j-1}\left(d\frac{p_{j}}{q_{j}}\right)-q_{j-1}a-p_{j-1}d+p_{j-1}c\right)\qquad\left(\textnormal{using eq. \eqref{eq-1-0}}\right)\nonumber\\
=&-\gamma_{j}(b-a,d-c)+(-1)^{\left\lfloor\frac{j}{2}\right\rfloor}(-1)^{j}\left(\frac{d}{q_{j}}\left(p_{j}q_{j-1}-q_{j}p_{j-1}\right)+p_{j-1}c-q_{j-1}a\right)\nonumber\\
=&-\gamma_{j}(b-a,d-c)+(-1)^{\left\lfloor\frac{j}{2}\right\rfloor}(-1)^{j}\left(\frac{d}{q_{j}}(-1)^{j}+p_{j-1}c-q_{j-1}a\right)\qquad\left(\textnormal{using lemma \eqref{lemma-3}\textnormal{, point }}(ii)\right)\nonumber\\
=&-\gamma_{j}(b-a,d-c)+(-1)^{\left\lfloor\frac{j}{2}\right\rfloor}\frac{d}{q_{j}}+(-1)^{\left\lfloor\frac{j}{2}\right\rfloor}(-1)^{j}\left(p_{j-1}c-q_{j-1}a\right)\nonumber\\
=&-\gamma_{j}(b-a,d-c)+(-1)^{\left\lfloor\frac{j}{2}\right\rfloor}\mathrm{sgn}(d)+(-1)^{\left\lfloor\frac{j}{2}\right\rfloor}(-1)^{j}\left(p_{j-1}c-q_{j-1}a\right)\qquad\left(\textnormal{using eq. \eqref{eq-6} and \eqref{eq-7}}\right)\nonumber\\
=&-(-1)^{\left\lfloor\frac{j}{2}\right\rfloor}\mathrm{sgn}(d)\mathrm{det}(M)+(-1)^{\left\lfloor\frac{j}{2}\right\rfloor}\mathrm{sgn}(d)+(-1)^{\left\lfloor\frac{j}{2}\right\rfloor}(-1)^{j}\left(p_{j-1}c-q_{j-1}a\right)\qquad\left(\textnormal{using eq. \eqref{eq-9}}\right)\nonumber\\
=&(-1)^{\left\lfloor\frac{j}{2}\right\rfloor}\mathrm{sgn}(d)\left(1-\mathrm{det}(M)+(-1)^{j}\mathrm{sgn}(d)\left(p_{j-1}c-q_{j-1}a\right)\right)\qquad\left(\textnormal{as }\left(\mathrm{sgn}(d)\right)^{2}=\mathrm{sgn}(d)\right)
\end{align}
\end{small}
Putting equations \eqref{eq-7}, \eqref{eq-9} and \eqref{eq-10} together, we found:
\begin{equation}\label{eq-11}
P_{j}=(-1)^{\left\lfloor\frac{j}{2}\right\rfloor}\mathrm{sgn}(d)\begin{pmatrix}
1&1-\mathrm{det}(M)+(-1)^{j}\mathrm{sgn}(d)\left(p_{j-1}c-q_{j-1}a\right)\\
0&\mathrm{det}(M)
\end{pmatrix}
\end{equation}
Let's write $b_{j}:=(-1)^{j}\mathrm{sgn}(d)\left(p_{j-1}c-q_{j-1}a\right)$, we get:
\begin{enumerate}[$(1)$]
\item $M\in SL_{2}\left(\mathbb{Z}\right)\implies\mathrm{det}\left(M\right)=1$; then, using lemma \eqref{lemma-1}, equation \eqref{eq-11} becomes:
\begin{equation}\label{P+}
P_{j}^{+}=(-1)^{\left\lfloor\frac{j}{2}\right\rfloor}\mathrm{sgn}(d)\begin{pmatrix}
1&b_{j}\\
0&1
\end{pmatrix}=(-1)^{\left\lfloor\frac{j}{2}\right\rfloor}\mathrm{sgn}(d)A^{b_{j}}
\end{equation}
\item $M\in GL_{2}\left(\mathbb{Z}\right)\setminus SL_{2}\left(\mathbb{Z}\right)\implies\mathrm{det}\left(M\right)=-1$; note that, $\forall n\in\mathbb{Z}$, $C A^{n}=\left(\begin{smallmatrix}1&0\\0&-1\end{smallmatrix}\right)\left(\begin{smallmatrix}1&n\\0&1\end{smallmatrix}\right)=\left(\begin{smallmatrix}1&n\\0&-1\end{smallmatrix}\right)$. Therefore, equation \eqref{eq-11} becomes:
\begin{equation}\label{P-}
P_{j}^{-}=(-1)^{\left\lfloor\frac{j}{2}\right\rfloor}\mathrm{sgn}(d)\begin{pmatrix}
1&2+b_{j}\\
0&-1
\end{pmatrix}=(-1)^{\left\lfloor\frac{j}{2}\right\rfloor}\mathrm{sgn}(d) C A^{2+b_{j}}
\end{equation}
\end{enumerate}
If we want to put equations \eqref{P+} and \eqref{P-} together, we note that $\frac{1-\mathrm{det}\left(M\right)}{2}= 0$ when $M\in SL_{2}\left(\mathbb{Z}\right)$ and $\frac{1-\mathrm{det}\left(M\right)}{2}= 1$ when $M\in GL_{2}\left(\mathbb{Z}\right)\setminus SL_{2}\left(\mathbb{Z}\right)$. Therefore,
\begin{align}\label{P}
P_{j}=&(-1)^{\left\lfloor\frac{j}{2}\right\rfloor}\mathrm{sgn}(d)\left(C A^{2}\right)^{\frac{1-\mathrm{det}\left(M\right)}{2}}A^{b_{j}}\nonumber\\
=&\left(A B^{-1}A\right)^{1-(-1)^{\left\lfloor\frac{j}{2}\right\rfloor}\mathrm{sgn}(d)}\left(C A^{2}\right)^{\frac{1-\mathrm{det}\left(M\right)}{2}}A^{b_{j}}\qquad\left(\textnormal{using equation \eqref{factor}}\right)
\end{align}
\item Recall equation \eqref{eq-5}; we have, by definition, $P_{0}=A^{-1}M A^{-1}B$. By direct calculation, we get:
\begin{equation}\label{eq-13}
P_{0}:=\begin{pmatrix}
b-d&b+c-(a+d)\\
d&d-c
\end{pmatrix}
\end{equation}
By induction on $k\geq 1$, we will show that $P_{k}=B^{-1}A^{2+(-1)^{k}n_{k}}P_{k-1}$, $\forall k\in\llbracket 1,\,j\rrbracket$.
\begin{enumerate}[$\bullet$]
\item Let $k:=1$; on one side, we have:
\begin{small}
\begin{align}\label{eq-14}
B^{-1}A^{2+(-1)^{1}n_{1}}P_{0}&=\begin{pmatrix}
1 & 0\\
-1 & 1
\end{pmatrix}\begin{pmatrix}
1 & 2-n_{1}\\
0 & 1
\end{pmatrix}\begin{pmatrix}
b-d & c+b-a-d \\
d & d-c
\end{pmatrix}=\begin{pmatrix}
1 & 2- n_{1} \\
-1 & n_{1}-1
\end{pmatrix}\begin{pmatrix}
b-d & c+b-a-d \\
d & d-c
\end{pmatrix}\nonumber\\
&=\begin{pmatrix}
b+d-d n_{1} & -c+d+b-a+(c-d)n_{1} \\
-b+d n_{1} & -b+a+(d-c) n_{1}
\end{pmatrix}
\end{align}
\end{small}
On the other side, 
\begin{align}\label{eq-15}
P_{1}&=\begin{pmatrix}
\alpha_{1} & \beta_{1} \\
\gamma_{1} & \delta_{1}
\end{pmatrix}\nonumber\\
&=(-1)^{\left\lfloor\frac{1}{2}\right\rfloor}\begin{pmatrix}
\left(q_{1}b-p_{1}d+(-1)^{1}\left(q_{0}b-p_{0}d\right)\right) & \left(q_{1}(b-a)-p_{1}(d-c)+(-1)^{1}\left(q_{0}(b-a)-p_{0}(d-c)\right)\right)\\
\left(p_{1}d-q_{1}b\right) & \left(p_{1}(d-c)-q_{1}(b-a)\right)\end{pmatrix}\nonumber\\
&=\begin{pmatrix}
b-n_{1}d+d & b-a-n_{1}(d-c)+(d-c) \\
n_{1}d -b & n_{1}(d-c)-(b-a)
\end{pmatrix}\qquad\left(\textnormal{using }\left(\begin{smallmatrix}
p_{0} & p_{1} \\
q_{0} & q_{1}
\end{smallmatrix}\right)=\left(\begin{smallmatrix}
1 & n_{1} \\
0 & 1
\end{smallmatrix}\right)\right)
\end{align}
The initialisation of the induction is valid as equations \eqref{eq-14} and \eqref{eq-15} are the same.
\item Suppose that $P_{k}=B^{-1}A^{2+(-1)^{k}n_{k}}P_{k-1}$ is true for $k>1$; we will show that it remains true for $k+1$:
\begin{small}
\begin{align*}
B^{-1}A^{2+(-1)^{k+1}n_{k+1}}P_{k}&=\begin{pmatrix}
1 & 0\\
-1 & 1
\end{pmatrix}\begin{pmatrix}
1 & 2+(-1)^{k+1}n_{k+1}\\
0 & 1
\end{pmatrix}P_{k}\\
&=\begin{pmatrix}
1 & 0\\
-1 & 1
\end{pmatrix}\begin{pmatrix}
1 & 2+(-1)^{k+1}n_{k+1}\\
0 & 1
\end{pmatrix}\begin{pmatrix}
\alpha_{k}(b,d)&\alpha_{k}(b-a,d-c)\\
\gamma_{k}(b,d)&\gamma_{k}(b-a,d-c)
\end{pmatrix}\qquad\left(\textnormal{by inductive hyp.}\right)\\
&=\begin{pmatrix}
1 & 2+(-1)^{k+1}n_{k+1}\\
-1 & -1+(-1)^{k+2}n_{k+1}
\end{pmatrix}\begin{pmatrix}
\alpha_{k}(b,d)&\alpha_{k}(b-a,d-c)\\
\gamma_{k}(b,d)&\gamma_{k}(b-a,d-c)
\end{pmatrix}:=\begin{pmatrix}
s & t\\
u & v
\end{pmatrix}
\end{align*}
\end{small}
We will show that $\left(\begin{smallmatrix}
s & t\\
u & v
\end{smallmatrix}\right)=\left(\begin{smallmatrix}
\alpha_{k+1}(b,d)&\alpha_{k+1}(b-a,d-c)\\
\gamma_{k+1}(b,d)&\gamma_{k+1}(b-a,d-c)
\end{smallmatrix}\right)$:
\begin{small}\begin{align*}
s&=1\cdot\alpha_{k}(b,d)+\left(2+(-1)^{k+1}n_{k+1}\right)\gamma_{k}(b,d)\\
&=(-1)^{\left\lfloor\frac{k}{2}\right\rfloor}\left(q_{k}b-p_{k}d+(-1)^{k}\left(q_{k-1}b-p_{k-1}d\right)
+\left(2+(-1)^{k+1}n_{k+1}\right)\left(p_{k}d-q_{k}b\right)\right)\\
&=(-1)^{\left\lfloor\frac{k}{2}\right\rfloor}\left(q_{k}b-p_{k}d+(-1)^{k}q_{k-1}b-(-1)^{k}p_{k-1}d+2p_{k}d-2q_{k}b+(-1)^{k+1}n_{k+1}p_{k}d-(-1)^{k+1}n_{k+1}q_{k}b\right)\\
&=(-1)^{\left\lfloor\frac{k}{2}\right\rfloor}\left(p_{k}d-q_{k}b+(-1)^{k+1}d\left(n_{k+1}p_{k}+p_{k-1}\right)-(-1)^{k+1}b\left(n_{k+1}q_{k}+q_{k-1}\right)\right)\\
&=(-1)^{\left\lfloor\frac{k}{2}\right\rfloor}\left(p_{k}d-q_{k}b+(-1)^{k+1}\left(d p_{k+1}-b q_{k+1}\right)\right)\qquad\left(\textnormal{using lemma \eqref{lemma-3}\textnormal{, point }}(i)\right)\\
&=(-1)^{\left\lfloor\frac{k}{2}\right\rfloor}(-1)^{k}\left((-1)^{k}\left(p_{k}d-q_{k}b\right)+\left(b q_{k+1}-d p_{k+1}\right)\right)\\
&=(-1)^{\left\lfloor\frac{k+1}{2}\right\rfloor}\left(q_{k+1}b-p_{k+1}d-(-1)^{k}\left(q_{k}b-p_{k}d\right)\right) \qquad\left(\textnormal{using lemma \eqref{lemma-4}}\right)\\
&=(-1)^{\left\lfloor\frac{k+1}{2}\right\rfloor}\left(q_{k+1}b-p_{k+1}d+(-1)^{k+1}\left(q_{k}b-p_{k}d\right)\right)\\
&=\alpha_{k+1}(b,d)
\end{align*}\end{small}
From this, we get directly:
\begin{align*}
t&=1\cdot\alpha_{k}(b-a,d-c)+\left(2+(-1)^{k+1}n_{k+1}\right)\gamma_{k}(b-a,d-c)\\
&=\alpha_{k+1}(b-a,d-c)
\end{align*}
Then, 
\begin{align*}
u&=(-1)\cdot\alpha_{k}(b,d)+\left(-1+(-1)^{k+2}n_{k+1}\right)\gamma_{k}(b,d)\\
&=(-1)^{\left\lfloor\frac{k}{2}\right\rfloor}\left(-q_{k}b+p_{k}d+(-1)^{k+1}\left(q_{k-1}b-p_{k-1}d\right)
+\left(-1+(-1)^{k+2}n_{k+1}\right)\left(p_{k}d-q_{k}b\right)\right)\\
&=(-1)^{\left\lfloor\frac{k}{2}\right\rfloor}\left(p_{k}d-q_{k}b+(-1)^{k+1}\left(q_{k-1}b-p_{k-1}d\right)-p_{k}d+q_{k}b-(-1)^{k+1}n_{k+1}p_{k}d+(-1)^{k+1}q_{k}b n_{n+1}\right)\\
&=(-1)^{\left\lfloor\frac{k}{2}\right\rfloor}(-1)^{k+1}\left(b\left(n_{k+1}q_{k}+q_{k-1}\right)-d\left(n_{k+1}p_{k}+p_{k-1}\right)\right)\\
&=(-1)^{\left\lfloor\frac{k+1}{2}\right\rfloor}\left(d p_{k+1}-b q_{k+1}\right)\qquad\left(\textnormal{using lemma \eqref{lemma-3}\textnormal{, point }}(i)\textnormal{ and lemma \eqref{lemma-4}}\right)\\
&=\gamma_{k+1}(b,d)
\end{align*}
Finally, using above calculation for $u$:
\begin{align*}
v&=(-1)\alpha_{k}(b-a,d-c)+\left(-1+(-1)^{k+2}n_{k+1}\right)\gamma_{k}(b-a,d-c)\\
&=\gamma_{k+1}(b-a,d-c)
\end{align*}
\end{enumerate}
\end{enumerate}
We just showed:
\begin{align*}
P_{j}&=\left(B^{-1}A^{2+(-1)^{j}n_{j}}\right)\left(B^{-1}A^{2+(-1)^{j-1}n_{j-1}}\right)\cdots\left(B^{-1}A^{2+(-1)^{1}n_{1}}\right)P_{0}\nonumber\\
&=\Big(\prod_{k=1}^{j}B^{-1}A^{2+(-1)^{j+1-k}n_{j+1-k}}\Big)P_{0}
\end{align*}
Using equation \eqref{P} and the definition of $P_{0}$, we get:
\begin{equation}
\left(A B^{-1}A\right)^{1-(-1)^{\left\lfloor\frac{j}{2}\right\rfloor}\mathrm{sgn}(d)}\left(C A^{2}\right)^{\frac{1-\mathrm{det}\left(M\right)}{2}}A^{b_{j}}=\Big(\prod_{k=1}^{j}B^{-1}A^{2+(-1)^{j+1-k}n_{j+1-k}}\Big)A^{-1}M A^{-1}B
\end{equation}
Solving this for $M$, we obtain formula \eqref{eq-1}. Note that we made, in above development, no assumptions on the continued fraction's length $j$; this shows that formula \eqref{eq-1} is independant of the chosen representation of the continued fraction associated to the rational $\frac{b}{d}$. 
\end{proof}
As another example, we can retrieve the fact that $A^{n}=\left(\begin{smallmatrix}1&n\\0&1\end{smallmatrix}\right)$, $\forall n\in\mathbb{Z}$ from lemma \eqref{lemma-1} simply by applying formula \eqref{eq-1} to the matrix $\left(\begin{smallmatrix}1&n\\0&1\end{smallmatrix}\right)$. Here, $j:=1$ as $\frac{n}{1}=[n]$ and $b_{1}=(-1)^{1}\mathrm{sgn}(1)\left(p_{0}\cdot 0-q_{0}\cdot 1\right)=0$ (recall that $q_{0}:=0$). Thus, 
\begin{equation}
\begin{pmatrix}1&n\\0&1\end{pmatrix}=\underbrace{\left(A B^{-1}A\right)^{1-(-1)^{\left\lfloor\frac{1}{2}\right\rfloor}\mathrm{sgn}(1)}}_{=I}A A^{-\left(2-n\right)}\underbrace{B A^{0}B^{-1}}_{=I}A=A A^{n-2}A=A^{n}
\end{equation}

\end{document}